\documentclass[12 pt]{amsart}
\usepackage{amscd,amssymb}
\usepackage{array}
\oddsidemargin=0.3in \evensidemargin=0.3in \theoremstyle{plain}

\newtheorem{prop}[section]{Proposition}

\theoremstyle{definition} \theoremstyle{procedure}
\newtheorem{rem}[section]{Remark}
\newtheorem{defn}[section]{Definition}
\newtheorem{exmp}[section]{Example}
\newtheorem{alg}{Algorithm}

\begin{document}
\title{The F4 Algorithm For Euclidean Rings}
\author{Afshan Sadiq$^*$}
\address{*Abdus Salam School of Mathematical Sciences, GC University, Lahore, Pakistan}
\email{afshanatiq@gmail.com}

\begin{abstract}
This short note is the generalization of Faug\'{e}re F4-algorithm for polynomial rings with coefficients in Euclidean rings. This algorithm computes successively a Gr\"{o}bner basis replacing the reduction of one single s-polynomial in Buchberger's algorithm by the  simultaneous reduction of several polynomials.
\end{abstract}

\maketitle
The concept of Gr\"{o}bner Bases was introduced by Bruno Buchberger (1965) in the context
of his work on performing algorithmic computations in residue classes of polynomial rings.
Buchberger’s algorithm for computing Gr\"{o}bner Bases is a powerful tool for solving many
important problems in polynomial ideal theory. To compute Gr\"{o}bner bases efficiently there are many methods like FGLM, Gr\"{o}bner walk $($cf. \cite{DJ}$)$. Fauger\'{e} \cite{JC} introduced a new efficient method F4 algorithm to compute Gr\"{o}bner bases using linear algebra, and a selection strategy among the critical pairs which occurs in the computation of Gr\"{o}bner basis.
This algorithm computes successively a Gr\"{o}bner basis replacing the reduction of one single s-polynomial in Buchberger's algorithm by the  simultaneous reduction of several polynomials. In the book of Adams and Loustaunau $($cf. \cite{WP}$)$ the concept of Gr\"{o}bner bases over polynomial rings with coefficients in a ring is developed. The aim of this short note is to show that Faug\'{e}re's F4-algorithm works also in polynomial rings with coefficients in Euclidean rings.
\\ \\Let $>$ be a fixed global monomial ordering and R be an Euclidean ring.
\begin{defn}
We fix the following notations, writing $f\in R[x_1, \ldots , x_n]$, $f\neq 0$, in a unique way as a sum of non-zero terms
$$f=a_{\alpha_1} x^{\alpha_1} +a_{\alpha_2} x^{\alpha_2} +\ldots +a_{\alpha_s} x^{\alpha_s} , \, \,\,\,\,\,\,\,\,\, x^{\alpha_1} > x^{\alpha_2} > \ldots  > x^{\alpha_s},$$ and $a_{\alpha_1},a_{\alpha_2} ,\ldots ,a_{\alpha_s} \in R$. We call:
\begin{itemize}
\item[1.]$LM(f):=x^{\alpha_1}$, the leading  monomial of $f$,
\item[2.]$LE(f):={\alpha_1}$, the leading  exponent of $f$,
\item[3.]$LT(f):=a_{\alpha_1} x^{\alpha_1}$, the leading  term of  $f$,
\item[4.]$LC(f):=a_{\alpha_1}$, the leading  monomial of $f$,
\item[5.] We define the leading monomial and the leading term of $0$ to be $0$, and $0$ to be smaller than any monomial.
\item [6.] Let $G \subset R[x_1, \ldots , x_n]$, then $L(G):=\langle \{LT(g)\,| \, g\in G\}\rangle_{R[x_1, \ldots , x_n]}$ is the leading ideal of $G$.
\end{itemize}

\end{defn}

\begin{defn}
Let $I\subset R[x_1, \ldots , x_n]$ be an ideal. A finite set $G \subset R[x_1, \ldots , x_n]$ is called a Gr\"{o}bner basis of $I$ with respect to $>$ if $G \subset I$, and $L(I)=L(G)$.
\end{defn}

\begin{defn}
Let $H\subset R[x_1,\ldots, x_n]$ be finite set. $H$ is called interreduced if for all $p\neq q$, $p,q \in H, LT(p)$ does not divide $LT(q)$. Furthermore if $LM(p)| LM(q)$ then the remainder $LC(q)$ mod $LC(p)$ of $LC(q)$ with respect to the division of $LC(q)$ by $LC(p)$ in the Euclidean ring is $LC(q)$.
\end{defn}

The existence of a set of interreduced generators of $\langle H \rangle$ is given by the following algorithm.

\begin{alg}$Interreduce(H)$
\begin{itemize}
\item[Input]: $H$ a set of polynomials.
\item[Output]: $L$ a set of interreduced polynomials such that $\langle H\rangle=\langle L\rangle$.
\\ $todo = 1$
\item while $($ there exist $h, k \in H, h\neq k, LM(k)|LM(h)$ and $todo = 1)$
\\ \indent $todo = 0$
\item  \indent \qquad if $(LT(k)|LT(h))$
\\ \indent \qquad $h:=h - \frac{LT(h)}{LT(k)} k$;
\\ \indent \qquad $todo = 1$
\\  \indent \qquad if $(h=0)$
\\  \indent \quad  \qquad $H:=H\backslash \{h\}$;
 \item \indent \qquad  else
\\  \indent \qquad  if $(LC(k)$ mod $LC(h))\neq 0$;
\\   \indent \quad  \qquad  compute $c=gcd(LC(h), LC(k))=a LC(h) + b LC(k)$;
 \\   \indent  \quad \qquad if $(a$ is a unit $)$
\\  \indent \qquad \qquad $h:=a h+b \frac{LM(h)}{LM(k)} k$ ;
\\  \indent \qquad \qquad $todo = 1$
\\   \indent \quad  \qquad if $ (LC(h)$ mod $LC(k)\neq LC(h) )$;
\\  \indent \qquad  \qquad  $h:=h+ \frac{LC(h)\, mod\, LC(k)-LC(h)}{LC(k)} \cdot \frac{LM(h)}{LM(k)} \cdot k$;
\\  \indent \qquad \qquad $todo = 1$
 \item return$(H)$;
\end{itemize}
\end{alg}

\begin{alg}$F4$
\begin{itemize}
\item[Input]: $G$ set of polynomials, $S$ a selection strategy\footnote{A selection strategy $S$ associates to the pair set $P$ a subset $S(P)\subset P$. A trivial example is $S(P)=P$. In the case that $G$ is a set of homogeneous polynomial $S(P)=\{(f,g)\in P \, | \, deg(lcm(LM(f),LM(g)))$ is minimal $\}$ is a good choice.}.
\item[Output]: Gr\"{o}bner basis for $\langle G\rangle$.
 \item $G=Interreduce(G)$;
 \item $P:=\{(f,g)\, | \, f,g \in G\}$;
 \item $while(P\neq \emptyset)$
\\   \indent \quad $H:=\{\frac{cx^\alpha}{LT(g)}\cdot g,\,\, \frac{cx^\alpha}{LT(f)}\cdot f \, |\, (f,g)\in S(P), \, cx^\alpha =lcm(LT(f),LT(g))\}$;
\\ \indent \quad $P:=P\backslash S(P)$;
\\ \indent \quad $H:=Interreduce(H\cup G)$;
\\ \indent \quad $P:=P\cup \{(f,h)\, |\, f\in G, \, h\in H, LT(h)\notin \langle LT(g)\, |\, g\in G\rangle \}$;
\\ \indent \quad $G:=H$;
 \item return $(G)$;
\end{itemize}
\end{alg}

\begin{prop}
The algorithm $F4$ terminates and the result is a Gr\"{o}bner basis of the ideal generated by input.
\end{prop}
\begin{proof}
The termination is a consequence of the ring $R[x_1, \ldots, x_n]$ being noetherian. Each time the set $G$ is enlarged the leading ideal $L(G)$ is enlarged properly. This has to stop after finitely many steps.
\\ The result $G$ is a Gr\"{o}bner basis of the ideal generated by the input because it satisfies Buchberger's criterion $($cf. \cite{GP1}$)$. Nmely the normal form of the $s$-polynomial of any two elements of $G$ with respect to $G$ is zero. This holds because for $f, g\in G$, $E=\frac{cx^\alpha}{LT(f)}\cdot f$ and $F=\frac{cx^\alpha}{LT(g)}\cdot g$ with $cx^\alpha = lcm(LT(f),LT(g))$ have been interreduced with elements from $G$ and $spoly(f,g)=E-F$.
\end{proof}

\begin{exmp}
Let $R=\mathbb{Z}$ and consider the ideal $$I=\langle 2abcd-2,abc+2abd+acd+bcd,ab+bc+ad+cd,a+b+c+d\rangle $$ in $ \mathbb{Z}[a,b,c,d]$ and $>$ the degrerlex ordering.
 \\In a test-implementation of the $F4$ and Buchberger's algorithm in SINGULAR, in $F4$, $91$ additions of polynomials are needed while
the Buchberger's algorithm needs $375$ additions.
\\ And we get $G=\{
-2,
 a+b+c+d,
 -b^2+d^2,
 -b c^2+b c d+c^2 d,
  b c d^2+c d^3,
 -c d^4\}$ a Gr\"{o}bner basis of $I$.
\end{exmp}
\begin{rem}
We presented here the idea of Faug\'{e}re algorithm in its simplest form to make the principle understandable. For an implementation one should keep the pair set as small as possible using Buchberger's criterion $($ chain criterion, product criterion, $($cf. \cite{GP1}$)$.
\end{rem}

\noindent {\bf Acknowledgments.} Thanks to my PhD supervisor Prof Dr Gerhard Pfister for his constant support and
 valuable suggestions.


\begin{thebibliography}{10}
\bibitem {ATG} Alessandro Giovini, Teo Moran, Gianfranco Niesi, Lorenzo Robbiano, Carlo Traverse. One sugar cube, please, or Selection strategies in the Buchberger algorithm, in: S.M. Watt (Ed.), Proc. 1991 Internat. Symp. on Symbolic and Algebraic computation, ISSAC' 91, ACM, New York, 1991.
\bibitem {DJ} David A.Cox, John Little, Donal O'Shea: Using Algebraic geometry. Springer, 2nd ed., 2004.
\bibitem {GP1} G.M. Greuel and G.Pfister: A SINGULAR Introduction to Commutative Algebra, 2nd ed., Springer, 2008.
\bibitem {GP2} G.M. Greuel, G.Pfister and H.Sch\"{o}nemann: SINGULAR - A Computer Algebra System for Polynomial Computations. Free software under GNU General Public Licence (1990-to date).
\bibitem {JC} Jean-Charles Faug\'{e}re: A new efficient algorithm for computing Gr\"{o}bner bases (F4). Journal of Pure and Applied Algebra 139 (1999), 61-88.
\bibitem {HM} H.M.M\"{o}ller: On the construction of Gr\"{o}bner bases using syzygies. J.Symb.Comp. 6 (1988), 345-359.
\bibitem {WP} W.W.Adams and P.Loustaunau: An Introduction to Gr\"{o}bner bases. Graduate studies in mathematics, vol. 3, American Mathematical Scociety, 2003.

\end{thebibliography}
\end{document}